\documentclass[11pt]{article}     

\usepackage{graphicx}
\usepackage{url}
\usepackage{color}	

\usepackage[mathscr]{euscript}		

\usepackage{dsfont}

\usepackage{psfrag}			
\usepackage{hyperref}

\usepackage{amsmath}
\usepackage{amsthm}

\usepackage{amssymb}
\usepackage{caption} 

\usepackage{anysize}

\usepackage{enumerate}
\usepackage{enumitem}

\newcommand{\ind}{\mathds}
\newcommand{\floor}[1]{{\lfloor #1 \rfloor}}
\newcommand{\Z}{\ensuremath{\mathbb{Z}}}
\newcommand{\N}{\ensuremath{\mathbb{N}}}

\newcommand{\E}{\ensuremath{\mathbb{E}}}
\renewcommand{\P}{\ensuremath{\mathbb{P}}}

\newtheorem{theorem}{Theorem}[section]
\newtheorem{lemma}[theorem]{Lemma}
\newtheorem{corollary}[theorem]{Corollary}
\newtheorem{proposition}[theorem]{Proposition}
\newtheorem{remark}[theorem]{Remark}
 
\newtheorem{claim}[theorem]{Claim} 

\newtheorem{definition}[theorem]{Definition}

\DeclareMathOperator{\e}{e}
\DeclareMathOperator{\C}{cap}

\DeclareMathOperator{\Cov}{Cov}

\numberwithin{equation}{section}

\definecolor{Red}{rgb}{1,0,0}
\definecolor{Blue}{rgb}{0,0,1}
\definecolor{Olive}{rgb}{0.41,0.55,0.13}
\definecolor{Yarok}{rgb}{0,0.5,0}
\definecolor{Green}{rgb}{0,1,0}
\definecolor{MGreen}{rgb}{0,0.8,0}
\definecolor{DGreen}{rgb}{0,0.55,0}
\definecolor{Yellow}{rgb}{1,1,0}
\definecolor{Cyan}{rgb}{0,1,1}
\definecolor{Magenta}{rgb}{1,0,1}
\definecolor{Orange}{rgb}{1,.5,0}
\definecolor{Violet}{rgb}{.5,0,.5}
\definecolor{Purple}{rgb}{.75,0,.25}
\definecolor{Brown}{rgb}{.75,.5,.25}
\definecolor{Grey}{rgb}{.7,.7,.7}
\definecolor{Black}{rgb}{0,0,0}

\newcommand{\T}{\mathcal{T}}

\newcommand{\ignore}[1]{{}}

\renewcommand{\P}{\ensuremath{\mathbb{P}}}

\newcommand{\one}{{\mathchoice {1\mskip-4mu\mathrm l}
         {1\mskip-4mu\mathrm l}
         {1\mskip-4.5mu\mathrm l}
         {1\mskip-5mu\mathrm l}}}
\begin{document}

\title{Transience of the vacant set for near-critical random interlacements in high dimensions}


\author{\renewcommand{\thefootnote}{\arabic{footnote}}
Alexander\ Drewitz
\footnotemark[1]
\qquad
\renewcommand{\thefootnote}{\arabic{footnote}}
Dirk\ Erhard
\footnotemark[{2}]$\,\,^{,}$ \hspace{-.6em} \footnotemark[{3}]
}

\footnotetext[1]{
Department of Mathematics, Columbia University, RM 614, MC 4419,
2990 Broadway,
New York, NY 10027, USA,\\
{\sl drewitz@math.columbia.edu}
}
\footnotetext[2]{
Mathematical Institute, Leiden University, P.O.\ Box 9512,
2300 RA Leiden, The Netherlands,\\
{\sl erhardd@math.leidenuniv.nl}
}
\footnotetext[3]{
DE was supported by ERC Advanced Grant 267356 VARIS. 
}

\date{\today}
\maketitle


\begin{abstract}

The model of random interlacements 
is a one-parameter family $\mathcal I^u,$ $u \ge 0,$ of random subsets of $\Z^d,$ which locally describes the trace of simple random walk on a $d$-dimensional torus run
up to time $u$ times its volume.
Its complement, the so-called vacant set $\mathcal V^u$, has been shown  to undergo a non-trivial percolation phase-transition in $u;$
i.e., there exists 
$u_*(d) \in (0, \infty)$ such that for $u \in [0, u_*(d))$ the vacant set $\mathcal V^u$ contains a unique infinite connected component
$\mathcal V_\infty^u,$ while
for $u > u_*(d)$ it consists of finite connected components.
It is known \cite{SZ11,SZ11B} that $u_*(d) \sim \log d,$
 and in this article we show 
the existence of $u(d) > 0$ with
$\frac{u(d)}{u_*(d)} \to 1$ as $d \to \infty$ such that
$\mathcal V_\infty^{u}$ is transient for all $u \in [0, u(d)).$


\end{abstract}


\section{Introduction and the main result}
\label{S1}
\subsection{Introduction}
\label{S1.1}

The model of random interlacements has been introduced by Sznitman \cite{SZ10} as 
a family of random subsets of $\Z^d$ denoted by $\mathcal I^u,$ $u \ge 0,$ where $u$
plays the role of an intensity parameter.
It locally describes the trace of simple random walk on
the discrete torus $(\Z/ N\Z)^d$ run up to time $uN^d$ (see Windisch \cite{Wi-08} as well as Teixeira and Windisch \cite{TeWi-11}).
Using the inclusion-exclusion formula the distribution of the set $\mathcal I^u$ 
can be neatly characterized via the equalities
\begin{equation*}
\P[ K \cap \mathcal I^u = \emptyset] = e^{- u \, \mathrm{cap}(K)}, \quad \forall K \subset \subset \Z^d.
\end{equation*}
Here, $\mathrm{cap}(K)$ is used to denote the capacity of the set $K$ (see \eqref{eq:cap} for the definition of capacity).
In a more constructive fashion, random interlacements at level $u$ 
can also be obtained by considering the trace of the elements in the support of
a Poisson point process with intensity parameter $u\geq 0$, 
which itself
takes values in the space of locally finite measures on
 doubly infinite simple random walk trajectories modulo time shift (see Section \ref{S2.2} for further details).

This constructive definition already suggests that
 the model exhibits long range dependence, and indeed the asymptotics
\begin{equation} \label{eq:Corr}
\Cov(\ind{1}_{x \in \mathcal I^u},\ind{1}_{y \in \mathcal I^u}) \sim c(u) \vert x- y \vert_2^{-(d-2)},
\end{equation}
(and similarly for $\mathcal I^u$ replaced by $\mathcal V^u$) 
holds for $\vert x - y \vert_2 \to \infty,$ as can be deduced from (0.11) in \cite{SZ10}. 
As a consequence,
standard techniques from Bernoulli percolation
do not apply anymore.
For example, due to \eqref{eq:Corr}  Peierl's argument
and the van den Berg-Kesten inequality break down.
The long range dependence also entails that random interlacements neither stochastically dominates nor can
be dominated by Bernoulli percolation (cf. Remark 1.6 1) of \cite{SZ10}). Moreover from the constructive definition
of random interlacements alluded to above, one can infer that the model does not fulfill the finite energy property
(see Remark 2.2 3) of \cite{SZ10}).
 These features make the model both, more appealing and more complicated to investigate. 

During the past couple of years there has been intensive research on random interlacements. Basic properties such as e.g. the shift-invariance, ergodicity and
connectedness of $\mathcal I^u$ have been established in the seminal paper  \cite{SZ10}. 
Since then, one has obtained a deeper understanding of the geometry of random interlacements.
In fact, R\'ath and Sapozhnikov 
\cite{RaSa-11} have shown the transience for random interlacements $\mathcal I^u$ 
itself throughout the whole range 
of parameters $u \in (0,\infty).$ The same authors in
 \cite{RaSa-10}, as well as  Procaccia and Tykesson \cite{PrTy-11} have
shown by essentially different methods (using ideas from the field of potential theory on the one hand,
 and stochastic dimension 
on the other hand) that any two points of the set $\mathcal I^u$ can be connected
by using at most $\lceil d/2 \rceil$ trajectories from the constructive definition described above. Recently, using in parts extensions of the techniques in \cite{RaSa-10},
this result has been generalized to an arbitrary number of points by Lacoin and Tykesson \cite{LaTy-12}.
Another step in showing that the geometry of random interlacements resembles that of $\Z^d$ has been undertaken by 
{\v{C}}ern{\'y} and Popov \cite{CePo-12}, where
the authors prove that the chemical distance (also called graph distance or internal distance)
in the set $\mathcal I^u$ is comparable to that of $\Z^d.$ Using this result they proceed to prove a shape theorem
for balls
in $\mathcal I^u$ with respect to the metric induced by the chemical distance.

It is particularly interesting 
to obtain a deeper understanding of the vacant set $\mathcal V^u$ and its geometry also.
Indeed, on the one hand, this is more challenging than the investigation of $\mathcal I^u$ in the sense that
one cannot directly take advantage of the many tools
available for simple random walk, which have proven to be very helpful in understanding the set $\mathcal I^u.$
On the other hand, it has been shown by Sznitman  \cite{SZ10} as well as Sidoravicius and Sznitman \cite{SSZ09}
that there exists a non-trivial percolation phase-transition for $\mathcal V^u$ at some $u_*(d) \in (0,\infty)$ in the following sense:
For $u > u_*(d)$ the vacant set $\mathcal V^u$ as a subgraph of $\Z^d$
 contains only finite connected components (subcritical phase), whereas for $ u \in [0, u_*(d))$ it has an
 infinite connected component almost surely
(supercritical phase). 
Using a strategy inspired by that of the seminal paper of Burton and Keane \cite{BuKe-89}, and taking care of the difficulties arising from the lack of the finite energy property for random interlacements,
Teixeira \cite{Te-09} 
has shown the uniqueness of the 
infinite connected component of $\mathcal V^u$  (denoted by  $\mathcal V_\infty^u$) 
in the supercritical phase. 

While for random interlacements itself many results have been shown to be valid
for any $u > 0,$ the situation is more complicated 
when investigating the 
vacant set $\mathcal V^u.$ In fact, while there are few results concerning the vacant set in the first place so far,
the ones which describe geometric properties such as Teixeira \cite{Te-11}, Drewitz, R\'ath and Sapozhnikov \cite{DrRaSa-12}, 
Popov and Teixeira \cite{PoTe-12} (dealing with the size distribution of finite clusters of the vacant
set and local uniqueness properties of $\mathcal V_\infty^u$) and Drewitz, R\'ath and Sapozhnikov
\cite{DrRaSa-12b} as well as Procaccia, Rosenthal and Sapozhnikov \cite{PrRoSa-13} (providing 
chemical distance results as well as heat kernel estimates in a more general context) are valid for some non-degenerate
fraction of the supercritical phase only.
To the best of  our knowledge, our main result
Theorem \ref{thm:mainres} is the first one concerning geometric properties of the vacant set which is valid throughout most, and asymptotically all, of the supercritical
phase for $\mathcal V^u.$

\bigskip

\subsection{Main result}
\label{S1.2}
Here we formulate our main result.
For this purpose recall that a connected graph with finite degree $G=(V,E)$ with vertex set 
$V$ and edge set $E$ is called transient if simple random walk 
on $G$ is transient. For the rest of this article $V$ will usually denote a subset of $\Z^d$ and $E$
will be the set of nearest neighbor edges in $\Z^d$ which have both ends contained in $V.$
\begin{theorem}
\label{thm:mainres}
Let $\varepsilon \in (0,1)$. There is $d_0=d_0(\varepsilon)\in\N$, 
such that for all $d\geq d_0$ and all $u\leq (1-\varepsilon)u_{*}(d)$,
the unique infinite connected component $\mathcal{V}_{\infty}^{u}$
of the vacant set $\mathcal{V}^{u}$ of random interlacements in $\Z^d$ is transient $\P$-a.s.
\end{theorem}
Recall here that  $u_*(d) \sim \log d,$ see \cite{SZ11,SZ11B}, where $\log$ denotes the natural logarithm. We refer to Section \ref{S2} for a rigorous definition of the terms
appearing in Theorem \ref{thm:mainres}.

\subsection{Discussion}
\label{S1.3}
Theorem \ref{thm:mainres} provides a rough geometrical description
of the infinite connected component of the vacant set, which is valid throughout most of the supercritical 
phase when $d$ is large enough.
To establish this result we introduce a classification of vertices in $\Z^3 \times \{0\}^{d-3}$ into ``good''
ones and ``bad''
ones, where ``good'' refers to having good local connectivity properties.
This way the problem will be reduced
to showing the transience of an infinite connected component of good vertices in $\Z^3.$
Our construction of this infinite cluster will
employ
results of Sznitman \cite{SZ11,SZ12}, whereas the proof of
the actual transience of this component uses ideas of Angel, Benjamini, Berger and Peres \cite{AnBeBePe-06}.
Besides making the attempt to extend our result to the entire supercritical phase 
it would be interesting to obtain a more precise
understanding of $\mathcal{V}_{\infty}^{u}$.
Results in this direction have been obtained in \cite{DrRaSa-12b,PrRoSa-13}.
A key assumption in these papers was  a  local uniqueness property (in our context of
$\mathcal{V}_{\infty}^{u}$), which roughly states that with high probability the second largest component
in a predetermined macroscopic box is small compared to the largest connected component in the same box.
However, this local uniqueness property has so far only been established for a non-degenerate 
part of the supercricital phase, and obtaining its validity throughout the whole supercritical phase
would be an interesting topic for further investigations.
\medskip

The rest of this article is organized as follows. In Section \ref{S2} we introduce
further notation, give a more detailed description of the model and provide a decoupling inequality
tailored to our needs (Proposition \ref{prop:decoupl}). 
The proof of Theorem \ref{thm:mainres} is carried out in Section \ref{S3}.  Sections \ref{S4} and \ref{S5}
contain the proofs of auxiliary results employed when proving Theorem \ref{thm:mainres}.


\section{Notation and introduction to the model}
\label{S2}
Section \ref{S2.1} introduces notation used in this article,
Section \ref{S2.2} defines random interlacements, 
while Section \ref{S2.3} states
a decoupling inequality.
Throughout the article we assume that $d\geq 3$.

\subsection{Basic notation}
\label{S2.1}
In the rest of this article we will tacitly identify $\Z^3$ with $\Z^3\times\{0\}^{d-3}$
via the bijection $(x_1, x_2, x_3) \mapsto (x_1, x_2, x_3, 0, \ldots, 0),$ if no confusion arises.
\\

For a subset $K\subseteq \Z^d$ we write $K\subset\subset \Z^d$, if its cardinality 
$|K|$ is finite or equivalently if $K$ is compact.
We denote by $|\cdot|_1$ the $\ell^1$-norm, by $|\cdot|_2$ the Euclidean norm, whereas $|\cdot|_\infty$ 
stands for the $\ell^\infty$-norm on $\Z^d$.
Sites $x,x'$ in $\Z^d$ are said to be nearest neighbors 
($*$-neighbors), 
if $|x-x'|_1=1$ 
($|x-x'|_\infty=1$).
A sequence $x_0,x_1,\ldots, x_n$ in $\Z^d$ is called a 
nearest neighbor path ($*$-path), if
$x_i$ and $x_{i+1}$ are nearest neighbors 
($*$-neighbors), for all $0\leq i\leq n-1$; in this case we say that the path has length
$n+1.$
A set $K\subseteq \Z^d$ is said to be connected ($*$-connected), if for any pair $x_1, x_2\in K$ 
 there exists
a nearest neighbor ($*$-neighbor) path $x_1, y_1, y_2, \ldots, y_n, x_2$ 
such that these vertices are contained in $K.$
For $K\subseteq \Z^d$ we introduce the following notions of boundaries
\begin{align}
\label{eq:boundary}
\begin{split}
\partial_{\mathrm{int}} K&= \{x\in K:\, \mbox{$x$ has a nearest neighbor in $K^c$}\},\\
\partial_{\mathrm{int}}^{*}K &= \{x\in K:\, \mbox{$x$ has a $*$-neighbor in $K^c$}\},\\
\partial K &= \{x\in K^c:\, x\mbox{ has a nearest neighbor in $K$}\},\\
\partial^{*}K &= \{x\in K^c:\, x\mbox{ has a $*$-neighbor in $K$}\},\\
\end{split}
\end{align}

to which we refer as interior boundary (interior $*$-boundary) and boundary ($*$-boundary),
respectively.
Moreover, the exterior  boundary (exterior $*$-boundary),
denoted by $\partial_{\mathrm{ext}} K$ ($\partial_{\mathrm{ext}}^{*} K$),
is the set of vertices in the boundary ($*$-boundary),
which are the starting point of an infinite non-intersecting nearest neighbor path
with no vertex inside $K$.\\
The closure of a set $K\subseteq \Z^d$ is defined by
$\overline{K} = K \cup \partial K$.
If $x\in\Z^d$ or $x\in\Z^3$ and $L\geq 0$, we write 
\begin{equation*}
\begin{aligned}
&B_i(x,L) = \{y\in\Z^d:\ |x-y|_i\leq L\} \quad \text{ and } \quad  
&B_{i}^{3}(x,L) = \{y\in\Z^3:\ |x-y|_i\leq L\},
\end{aligned}
\end{equation*}
respectively, for $i\in\{1,2,\infty\}.$
Given a set $K\subseteq \Z^d$ and $w:\N_0\to \Z^d$,
we denote by
\begin{equation}
\label{eq:hittime}
\begin{aligned}
H_K(w) = \inf\{n\geq 0\colon\, w(n)\in K\} \quad \text{ and } \quad
\widetilde{H}_K(w) = \inf\{n\geq 1\colon\, w(n)\in K\}
\end{aligned}
\end{equation}
the entrance time in and the hitting time of $K$, respectively.
For $x\in\Z^d$, let $P_x$ denote the law of simple random walk
on $\Z^d$ with starting point $x$.
If $K\subset\subset \Z^d$, we write $\e_K$ for the equilibrium 
measure of $K$, i.e.,
\begin{equation}
\label{eq:cap}
\e_K(x) = P_x\big[\widetilde{H}_K = \infty\big]\one_{\{x\in K\}} \quad \text{ and } \quad
\C(K) = \sum_{x\in K}e_K(x)
\end{equation}
for the total mass of  $e_K,$ which is usually referred to as the capacity of $K.$
From this one immediately obtains the subadditivity of the capacity; i.e., for all $K,K' \subset \subset \Z^d$
one has
\begin{equation} \label{eq:subadd}
\C(K\cup K') \le \C(K) + \C(K').
\end{equation}
We denote by $g:\Z^d\times\Z^d\rightarrow [0,\infty)$ the 
Green function of simple random walk on $\Z^d$, which is defined via
\begin{equation*}
g(x,x') = \sum_{n\in\N_0} P_x\big[X_n=x'\big], \quad \mbox{ for } x,x' \in\Z^d,
\end{equation*}
and we write $g(0)=g(0,0)$.
Finally, let us explain the convention we use concerning constants.
Throughout the article, small letters such as $c$, $c'$, $c_1,c_2,\cdots$,
denote constants which are independent of $d$.
Capital letters, such as $C$ and $C_1$ might depend on the dimension.
Constants that come with an index are fixed from their first appearance on (modulo changes of the dimension if they are capital letter
constants),
whereas constants without index may change from place to place.

\subsection{Definition of random interlacements}
\label{S2.2}
The model of random interlacements has been introduced in \cite{SZ10}, 
and we refer to this source for a discussion that goes beyond the description we are giving here. 
We write
\begin{equation*}
\label{eq:W+}
W_{+}= \Big\{w: \N_0\rightarrow \Z^d \colon\, 
|w(n)-w(n+1)|_1 = 1 \; \forall \,  n\in\N_0, \text{ and } \lim_{n\to\infty}|w(n)|_1=\infty\Big\}
\end{equation*}
for the set  of infinite nearest neighbor paths tending 
to infinity and
\begin{equation*}
W=\Big\{w:\Z\rightarrow\Z^d\colon\,
|w(n)-w(n+1)|_1=1 \; \forall \,  n\in\Z, \text{ and }  \lim_{n\to \pm \infty}|w(n)|_1
=\infty\Big\}
\end{equation*}
for the set of doubly infinite nearest neighbor paths
tending to infinity at positive and negative infinite times.
$W_+$ is endowed with the $\sigma$-algebra $\mathcal{W}_+$
generated by the canonical coordinate maps $X_n$, $n\in\N_0$.
Similarly, we will write $\mathcal{W}$ and
$X_{n}$, $n\in\Z$, 
for the canonical $\sigma$-algebra and the canonical coordinate process
on $W$.
We denote by $W^*$ the space of equivalence classes of 
trajectories in $W$ modulo time-shifts, i.e.,
\begin{equation*}
W^*= W/\thicksim, \mbox{ where } w\thicksim w' \mbox{ iff } w(\cdot) = w'(\cdot+k)
\mbox{ for some }k\in\Z.
\end{equation*}
We let $\pi^*:W\rightarrow W^*$ be the canonical projection
and endow $W^*$ with the $\sigma$-algebra induced by $\pi^*$ via
\begin{equation*}
\mathcal{W}^*= \big\{A\subset W^*\colon\, (\pi^*)^{-1}(A)\in \mathcal{W}\big\}.
\end{equation*}
We furthermore introduce for $K\subset\subset \Z^d$ the subsets
\begin{equation*}
\begin{aligned}
&W_K= \big\{w\in W\colon\, \mbox{there is $k\in\Z$
such that } w(k)\in K\big\},\\
&W_K^{*}= \pi^*(W_K)
\end{aligned}
\end{equation*}
of $W$ and $W^*,$ respectively.
Note that $W_K\in\mathcal{W}$ and $W_K^*\in \mathcal{W}^{*}$.
For $A,B\in \mathcal{W}_+$, $K\subset\subset\Z^d$ 
and $x\in\Z^d$ we define a finite measure $Q_K$ on $W$ via
\begin{equation*}
Q_K\big[(X_{-n})_{n\geq 0} \in A, X_0=x,
(X_n)_{n\geq 0} \in B\big]
= P_x[A\mid \widetilde{H}_K=\infty]\e_K(x)P_x[B].
\end{equation*}
According to Theorem 1.1 in \cite{SZ10} there exists a unique 
$\sigma$-finite measure $\nu$ on $(W^*,\mathcal{W}^*)$ 
such that for all $K\subset\subset \Z^d$ and 
$E\in \mathcal{W}^{*}$ with $E\subseteq W_K^*,$
 the equation
\begin{equation*}
\nu[E]= Q_K\big[(\pi^*)^{-1}(E)\big]
\end{equation*}
is fulfilled.
We will also need the space
\begin{equation*}
\begin{aligned}
\Omega =\Big\{\omega = \sum_{i\geq 0}& \delta_{(w_i^*,u_i)}
\mbox{ with } (w_i^*,u_i)\in W^*\times[0,\infty), \mbox{ for }
i\geq 0,\\
& \mbox{and } \omega[W_K^*\times[0,u]]<\infty 
\mbox{ for any }K\subset\subset \Z^d \mbox{ and } u\geq 0 \Big\}
\end{aligned}
\end{equation*}
of locally finite point measures on $W^* \times [0,\infty).$
Let $\mathcal{B}([0,\infty))$ be the Borel $\sigma$-algebra on
$[0,\infty)$ and let $\mathcal{A}$
be the $\sigma$-algebra on $\Omega$ which
is generated by the family of evaluation maps $\omega \mapsto \omega[D]$, 
$D\in\mathcal{W}^{*}\otimes \mathcal{B}([0,\infty))$.
We denote by $\P$ the law of the Poisson point process on $(\Omega, \mathcal{A})$
with intensity measure $\nu \otimes {\rm d}u$. This process is usually referred to as the interlacement
Poisson point process.
Random interlacements at level $u$
is then defined as the subset of $\Z^d$ given by
\begin{equation*}
\mathcal{I}^{u}(\omega) = \bigcup_{u_i\leq u}
\mathrm{range}(w_i^{*}), \quad \mbox{ where } 
\omega = \sum_{i\geq 0} \delta_{(w_i^*,u_i)} \in \Omega,
\end{equation*}
and $\mathrm{range}(w^*) = \big\{w(n): \, n\in\Z\big\}$
for arbitrary  $w\in \pi^{-1}(\{w^*\})$.
The vacant set at level $u \ge 0$ is defined by
\begin{equation*}
\mathcal{V}^{u}(\omega) = \Z^d\setminus\mathcal{I}^{u}(\omega), \quad
\omega \in\Omega.
\end{equation*}
As has been shown in \cite{SZ10} and \cite{SSZ09}, in  any dimension
$d \ge 3$ there exists a $u_*(d) \in (0,\infty)$ such that
for $u \in [0,u_*(d))$ the vacant set $\mathcal V^u$ contains an infinite connected component, whereas for
$u \in (u_*(d), \infty)$ it consists of finite connected components.

\subsection{Cascading events and a decoupling inequality}
\label{S2.3}

In this section we give a slightly refined version of a decoupling inequality 
of  \cite[Theorem 3.4]{SZ12}.
This is a fundamental tool to deal with the dependence structure inherent to the model.
Since the constants appearing in the decoupling inequality depend implicitly on the dimension $d,$
we have to pay special attention to their behavior for large $d.$
Proposition
\ref{prop:decoupl} below states all these dependencies explicitly.
We write $\Psi_x$, $x\in\Z^d$, for 
the canonical coordinates on $\{0,1\}^{\Z^d}$.
Let us recall Definition 3.1 of \cite{SZ12} of so-called cascading events.
\begin{definition}[Cascading events] 
\label{def:cascade}
Let $\lambda > 0$. A family $\mathcal{G} = (G_{x,L})_{x\in\Z^d,L\geq 1 \mbox{ integer}}$
of events in $\{0,1\}^{\Z^d}$ cascades with complexity at most $\lambda$,
if
\begin{equation*}
G_{x,L} \mbox{ is } \sigma\left(\Psi_{x'}, x'\in B_2(x,10\sqrt{d}L)\right)-\mbox{measurable
for each $x \in\Z^d$, $L\geq 1$},
\end{equation*}
and for each multiple $l$ of $100$, $x\in\Z^d$, $L\geq 1$,
there exists $\Lambda\subseteq\Z^d$ and a constant $C_1=C_1(\mathcal{G},\lambda)$
such that
\begin{equation*}
\begin{aligned}
&\Lambda\subseteq B_2(x,9\sqrt{d}lL),\\
&|\Lambda|\leq C_1l^{\lambda},\\
&G_{x,lL}\subseteq \bigcup_{x',x''\in\Lambda\,:\,|x'-x''|_2 \geq \frac{l}{100}\sqrt{d}L} 
G_{x',L}\cap G_{x'',L}.
\end{aligned}
\end{equation*}
\end{definition}

\begin{remark}
\label{rm:cascade}
Note that the cascading events are defined with respect to the
$\ell^2$-norm instead of the more common $\ell^\infty$-norm.
Since we are working in a high dimensional setting, this makes the
constants appearing in Proposition \ref{prop:decoupl} easier to control.
This again is due to the fact that, see (1.22) and (1.23) of \cite{SZ11B}, there are
constants $c_2,c_3>0$, which do not depend on $d$, such that for all $L\geq d$,
\begin{equation} \label{eq:capBd}
\bigg(\frac{c_2L}{\sqrt{d}}\bigg)^{d-2}\leq \C(B_2(0,L))
\leq \bigg(\frac{c_3L}{\sqrt{d}}\bigg)^{d-2}.
\end{equation}
\end{remark}
The notions introduced below pertain to the so-called sprinkling 
technique. 
The idea is that with high probability the 
 mutual dependencies of the
events under consideration can be dominated by considering  random interlacements at
two different levels $u_{\infty}^{-} < u$.
For this purpose we introduce for $l_0$ positive the quantity
\begin{equation}
\label{eq:fdecoupl}
f(l_0) = \prod_{k\geq 0}
\bigg(1+32e^2c_1^{d}
\frac{1}{(k+1)^{\frac{3}{2}}}l_0^{-(d-3)/2}\bigg).
\end{equation}
The constant $c_1$ will be chosen according to the formulation of Proposition \ref{prop:decoupl}
below.
Furthermore, we define for $u>0$, $u_{\infty}^{-}= u_{\infty}^{-}(u) = \frac{u}{f(l_0)}$, as well as for $L_0\geq 1$,
\begin{equation}
\label{eq:epsilondec}
\varepsilon(u)= \frac{2e^{-uL_{0}^{d-2}l_{0}}}
{1-e^{-uL_{0}^{d-2}l_{0}}}.
\end{equation}
We let $L_0\geq 1$ and define the scales $L_n= l_0^nL_0$, $n\in\N_0$. 
$L_n$ and $l_0$ will play from now on the role of $L$ and $l$ in Definition \ref{def:cascade}.
Finally, for a subset $A \subset \{0,1\}^{\Z^d}$ and $u\geq 0$ we write
\begin{equation*}
A^{u}:=\Big\{\omega\in\Omega:\ \one_{\mathcal{I}^{u}(\omega)}
\in A\Big\}.
\end{equation*}
Also, $A \subset \{0,1\}^{\Z^d}$ is called  increasing if the following holds:
For all $\xi \in A$ and $\xi' \in \{0,1\}^{\Z^d}$ such  that
$\xi_x \le \xi'_x$ holds for all $x \in \Z^d,$ one has that $\xi' \in A$ also.

A refinement of the arguments in \cite{SZ12}, proof of Theorem 2.6
(with a special emphasis on the dependence of the constants on the dimension), leads to
the following result.
\begin{proposition}[Decoupling inequality]
\label{prop:decoupl}
Let $\lambda >0$. Consider $\mathcal{G}= (G_{x,L})_{x\in\Z^d, L\geq 1\mbox{ integer}}$
a collection of increasing events on $\{0,1\}^{\Z^d}$
that cascades with complexity at most $\lambda$.
Then
there are $c_0,c_1>1$ (the latter one comes into play in (\ref{eq:fdecoupl})) such that for all $l_0\geq 10^6\sqrt{d}c_0,$ 
all $L_0\geq\sqrt{d}$ and all $n\in\N_0,$
one has
\begin{equation}
\label{eq:decoupl}
\sup_{x\in\Z^d} \P\Big[G_{x,L_n}^{u_{\infty}^{-}}\Big]
\leq \Big(C_1l_0^{2\lambda}\Big)^{2^n}
\Bigg(\sup_{x\in\Z^d}\P\Big[G_{x,L_0}^{u_0}\Big]+ \varepsilon(u_{\infty}^{-})\Bigg)^{2^n}.
\end{equation}
\end{proposition}
See Appendix \ref{A} for the proof.


\section{Proof of Theorem \ref{thm:mainres}}
\label{S3}

In this section we introduce a classification  of vertices in  $\Z^3$ into ``good''
(exhibiting good connectivity properties, see Definition \ref{def:Gyu} below) and ``bad'' vertices.
Subsequently, we give two auxiliary results on the existence of an infinite connected component 
of good vertices (Proposition
\ref{prop:infiniteGoodness}) which is transient as a subset of $\mathbb \Z^3$
(Proposition \ref{prop:transofG}).
From the latter result we deduce Theorem \ref{thm:mainres}.

\subsection{Auxiliary results}
\label{S3.1}
Let $\mathcal C_y=2y+\{0,1\}^d$, $y\in\Z^d$, and $\mathcal C = \mathcal C_0$.
\begin{definition}
\label{def:Gyu}
Let $u\ge 0$. A vertex $y \in \Z^3$ is defined to be u-good 
(with respect to $\omega \in \Omega$)
if
\begin{equation}
\label{eq:Gyu}
\begin{aligned}
\omega \in \mathcal G_{y,u}
:= \Bigg \{ \omega \in \Omega \, \colon \,  &\forall z \in \Z^3 \text{ with } | y-z |_1 \le 1, \; 
\text{ the set }\mathcal V^u(\omega) \cap \mathcal C_z
\text{ contains a connected}\\
& \text{component } 
\mathfrak C_{y,z} \text{ with } \vert \overline{{\mathfrak C}_{y,z}} \cap \mathcal C_z \vert \ge (1-d^{-2}) \vert \mathcal C_z \vert,
\text{ and these}\\
&\text{components are connected in } \mathcal V^u(\omega) \cap \Bigg( \bigcup_{ z\in\Z^3\, : \, | y-z |_1 \le 1} \mathcal C_{z} \Bigg) \Bigg\}.
\end{aligned}
\end{equation}
Otherwise, $y$ is called u-bad (with respect to $\omega \in \Omega$).
\end{definition}
\begin{remark}
\label{rm:unqiue}
Lemma 2.1 in \cite{SZ11} states that there is a $d_0\in\N$ such that if $d\geq d_0$, then 
any subset $V \subset \mathcal C$ contains at most one connected component
$\mathfrak C$ of $V$
such that $| \overline{\mathfrak C} \cap \mathcal C|\geq (1-d^{-2})|\mathcal C|$.
Thus, for $d\geq d_0$, if a connected component $\mathfrak C_{y,z}$ 
as in \eqref{eq:Gyu} exists, then it is necessarily unique.

\end{remark}
Denote by 
\begin{equation}
\label{eq:BG}
\mathcal G^u (\omega) := \{y \in \Z^3 : \omega \in \mathcal G_{y,u}\}
\quad \text{ and } \quad
\mathcal B^u (\omega) := \Z^3 \backslash \mathcal G^u(\omega)
\end{equation}
the  set of $u$-good  and $u$-bad vertices given $\omega.$
We can now state the auxiliary results alluded to above.
\begin{proposition}[Existence of an infinite connected component of good vertices]
\label{prop:infiniteGoodness}
Fix $\varepsilon \in (0,1)$.
There is $d_0=d_0(\varepsilon)\in\N$ such that for all $d\geq d_0$ and
$u\leq (1-\varepsilon)u_*(d)$, $\P$-a.s.
there exists an infinite connected component in $\mathcal G^u$. 
\end{proposition}

\begin{remark}
Using Proposition \ref{prop:badComponentsSmall} below,
it is not hard to establish the uniqueness of this infinite connected component. However,
since we do not need this uniqueness, we will not give a proof of this fact.
\end{remark}

In the forthcoming proposition all parameters are chosen according to Proposition 
\ref{prop:infiniteGoodness} above. From now on,  $\mathcal{G}_{\infty}^{u}$ will denote 
an arbitrary infinite connected component of $\mathcal G^u$.

\begin{proposition}[Transience of $\mathcal{G}_{\infty}^{u}$]
\label{prop:transofG}
Fix $\varepsilon \in (0,1)$.
There is $d_0=d_0(\varepsilon)\in\N$ such that for all $d\geq d_0$ and
$u\leq (1-\varepsilon)u_*(d)$, one has that $\mathcal{G}_{\infty}^{u}$ is transient $\P$-a.s.
\end{proposition}

The above two results will be proven in  Sections \ref{S4} and \ref{S5}.

\subsection{Proof of Theorem \ref{thm:mainres} given Propositions \ref{prop:infiniteGoodness} and
\ref{prop:transofG}}
\label{S3.2}
In this section we show how Theorem \ref{thm:mainres}
can be deduced from Propositions \ref{prop:infiniteGoodness}
and \ref{prop:transofG}.
For a connected subset $G$ of $\Z^d$,
let $\Pi_y (G),$ $y \in G$,  be the set of infinite
non-intersection (which we will also call simple for the sake of brevity) nearest neighbor paths on $G$ starting in $y$.
We recall the following characterization of the transience of $G.$
\begin{lemma} \label{lem:transienceCharact}
The following are equivalent:
\begin{enumerate} [label=(\alph{*}), ref=(\alph{*})]
 \item \label{item:Gtransient}
The graph $G$ (with connectivity structure induced by $\Z^d$) is transient.
\item \label{item:finiteExpVerticesIII}
There is $y\in G$ such that there is a probability measure $\mu$ on $\Pi_y(G)$ fulfilling
\begin{equation} \label{eq:finiteExpVerticesIII}
 \sum_{ x \in V} \mu^2 \big [\pi \in \Pi_y(G) :  x \in \pi \big ] < \infty.
\end{equation}
\end{enumerate}
\end{lemma}
\begin{remark}
A version similar to Lemma \ref{lem:transienceCharact} (in a more general context)
may be found in \cite[Chapter 2]{LyPe01} and \cite{Pe-99}. Note that in these two references the sum in (\ref{eq:finiteExpVerticesIII}) 
is taken over nearest neighbor edges  whose ends both lie in $G,$ rather than over the vertices.
However, using the fact that $\Z^d$ has uniformly bounded degree, one can deduce Lemma \ref{lem:transienceCharact} from the corresponding
edge-based version without problems. We will omit the proof of this fact.
\end{remark}
The strategy to prove Theorem \ref{thm:mainres} 
is as follows: Since by Propositions \ref{prop:infiniteGoodness} and \ref{prop:transofG} the subset
$\mathcal{G}_{\infty}^{u}$ of $\Z^3$
is transient for $d$ and $u$ as in the assumptions,
Lemma \ref{lem:transienceCharact} provides us with a measure $\mu$ on the simple nearest neighbor paths in $\mathcal G_\infty^u$
fulfilling \eqref{eq:finiteExpVerticesIII} for $G = \mathcal G_\infty^u.$ 
We then map simple nearest neighbor paths in $\mathcal G_\infty^u$
to simple nearest neighbor paths in $\mathcal V_\infty^u,$
in a way that does not blow up the lengths of the paths too much, cf. \eqref{eq:pitilde} below.
The pushforward of $\mu$ under  this mapping then supplies us with a probability measure 
supported on infinite simple
nearest neighbor paths in $\mathcal V_\infty^u$
that still satisfies condition \eqref{eq:finiteExpVerticesIII} (cf. Claim \ref{cl:phitilde}). We now make this strategy precise.

\begin{proof}[Proof of Theorem \ref{thm:mainres}]
We fix $\varepsilon\in(0,1)$ and choose $d_0=d_0(\varepsilon)$
such that the implications of both, Proposition \ref{prop:infiniteGoodness} and \ref{prop:transofG}, hold true.
Write
\begin{equation*}
\Gamma \colon\, \mathcal{G}_{\infty}^{u} \rightarrow \bigcup_{y \in \mathcal G_\infty^u} \mathfrak C_{y,y}
\end{equation*}
for the mapping
that sends $y \in \mathcal G_{\infty}^u$ to the element $z \in \mathfrak{C}_{y,y}$ that has minimal lexicographical order among all elements of
$\mathfrak{C}_{y,y}.$
Moreover, by the definition of $\mathcal G_\infty^u,$ for each 
\begin{equation} \label{eq:admissPair}
x, y \in \mathcal{G}_{\infty}^{u} \text{ with } |x-y|_1=1 
\end{equation}
there is
a simple nearest neighbor path $(\widetilde \pi_y^x(k))_{k=0}^{n-1}$
on 
\begin{equation*}
\mathcal{V}_\infty^{u}\cap
\{\mathcal C_z\, \colon\, z \in \Z^3 \text{ and } |y-z|_1\leq 1\}
\end{equation*}
such that \\
\begin{equation}
\label{eq:pitilde}
\begin{aligned}
& \bullet \widetilde  \pi_y^x(0) = \Gamma(y) \text{ and }
\widetilde  \pi_y^x(n-1) = \Gamma(x);\\
&\bullet n\leq   \Big \vert \bigcup_{z \in \Z^3 \, : \, \vert y - z \vert_1 \le 1	} 
 \mathfrak C_{y, z} \Big \vert \le 7\times2^{d}.
\end{aligned}
\end{equation}
For any pair of points $x$ and $y$ as in \eqref{eq:admissPair} we
choose and fix a path $\widetilde \pi_y^x$ with the above properties.
Given an infinite simple nearest neighbor path $\pi$
on $\mathcal{G}_{\infty}^{u}$, we obtain an infinite nearest neighbor path $\widetilde \pi$ on $\mathcal V_\infty^u$ starting in $\Gamma(\pi (0))$
by concatenating the paths $\widetilde \pi_{\pi(k)}^{\pi(k+1)},$ $k = 0, 1, 2, \ldots.$
Finally, we denote by $\varphi$ the map that sends $\pi$ to  the loop-erasure 
of $\widetilde \pi$ (note that the latter is an infinite simple nearest neighbor path in $\mathcal V_{\infty}^u$).

Now due to Proposition \ref{prop:transofG} and Lemma \ref{lem:transienceCharact} there exists a 
probability measure $\mu$ on $\Pi_y(\mathcal G_\infty^u)$, for some $y\in \mathcal G_\infty^u$,
fulfilling \eqref{eq:finiteExpVerticesIII}.
Hence, Theorem \ref{thm:mainres} is a consequence of the claim below and Lemma \ref{lem:transienceCharact}.
\end{proof}

\begin{claim}
\label{cl:phitilde}
If a measure $\mu$ on $\Pi_y(\mathcal G_\infty^u)$ fulfills \eqref{eq:finiteExpVerticesIII},
then so does the measure $\widetilde \mu := \mu \circ \varphi^{-1}$ on $\Pi_{\Gamma(y)}(\mathcal V_\infty^u).$
\end{claim}
\begin{proof}
In a slight abuse of notation, for $x \in \mathcal V_\infty^u \subset \Z^d$ define $\Gamma^{-1}(x)$ to be
 the $z \in \Z^3$ (unique, if it exists) such that $x \in \mathcal C_z.$ If no such $z$ exists, then let $\Gamma^{-1}(x) = \infty$ and define $|\infty-z|_1=\infty$
for $z\in\Z^3$. We will see that the latter case is of no importance,
since the construction of $\varphi$ is such that it restricts all paths on $\mathcal V_{\infty}^{u}$ to such $x$ for which $\Gamma^{-1}(x)\in\Z^3$. Then for 
$x \in \mathcal V_\infty^u,$ one has
$$
\mu \circ \varphi^{-1} [\pi \in \Pi_{\Gamma(y)}(\mathcal V_\infty^u)\, : \, x \in \pi ] \le
\sum_{z \in \Z^3 \, : \, \vert \Gamma^{-1}( x) - z \vert_1 \le 1} \mu [\pi \in \Pi_y(\mathcal G_\infty^u) \, : \, z \in \pi].
$$
Thus, an application of the Cauchy-Schwary inequality along with (\ref{eq:pitilde})
and the facts that $\vert \mathcal C_y \vert = 2^d$ and that for $x \in \Z^3$ one has
$\vert \{z \in \Z^3 \, : \, \vert z-x \vert_1 \le 1\} \vert = 7$, yields that
\begin{equation*}
\sum_{x\in\Z^d} \mu \circ \varphi^{-1} [\pi \in \Pi_{\Gamma(y)}(\mathcal V_\infty^u)\, : \, x \in \pi ]^2
\leq 2^{d} \cdot 7^2\sum_{z\in\Z^3} \mu^2[\pi \in \Pi_y(\mathcal G_\infty^u) \, : \, z \in \pi] < \infty,
\end{equation*}
from which we may conclude the claim.
\end{proof}

\section{Proof of Proposition \ref{prop:infiniteGoodness} (existence of an infinite connected component of good vertices)}
\label{S4}

In the proof of this proposition we 
exploit the fact that as $d \to \infty,$ certain averaging effects occur  which (in combination with so-called ``sprinkling'')
imply that with high probability and for slightly supercritical intensities $u,$ such hypercubes are $u$-good in the sense of Definition \ref{def:Gyu}
(a big chunk of this work is done by Theorem 4.2 in \cite{SZ11} and in Lemma \ref{lem:G} we neatly adapt this result to our purposes).
By identifying hypercubes with vertices this will lead to a dependent percolation problem on $\Z^3$. 
This is where we will take advantage of the decoupling inequality \eqref{eq:decoupl}
in order to deduce
that $*$-connected components of $u$-bad vertices are sufficiently small,
and hence an infinite connected  component of $u$-good vertices exists.

\subsection{Proof of Proposition \ref{prop:infiniteGoodness} given an auxiliary result}
\label{S4.1}
The result below provides an estimate on the size of $*$-connected components of $u$-bad vertices.
Its proof is postponed to Section \ref{S4.2}.
\begin{proposition}[$*$-connected components of $u$-bad vertices are small] 
\label{prop:badComponentsSmall}
Fix $\varepsilon \in(0,1)$. There is $d_0=d_0(\varepsilon)\in\N$ such that for all $d\geq d_0$, 
there are $C_2,C_3 > 0$ such that for all $u\leq (1-\varepsilon)u_*(d)$ and $N \in \N$
 \begin{equation*}
 \begin{aligned}
&\sup_{x \in \Z^3}  \P \big[ x \text{ is contained in a simple $*$-path 
of u-bad vertices of length at least } N \big]
\le C_2e^{-N^{C_3}}.
\end{aligned}
 \end{equation*}
\end{proposition}

Before we proceed, recall the notion of exterior boundary below (\ref{eq:boundary}).
We now prove Proposition \ref{prop:infiniteGoodness}.
\begin{proof}[Proof of Proposition \ref{prop:infiniteGoodness}]
For $x\in\Z^3$ define
\begin{align*}
\mathcal{G}_x = \left\{
 \begin{array}{ll}
 \text{the connected component of $u$-good vertices containing $x$,} & 
 \text{ if $x$ is $u$-good,}\\
 \emptyset, &\text{ otherwise.}
 \end{array}
 \right.
 \end{align*}
Now assume that there is $N\in\N$ such that $\mathcal{G}_x$ has  finite 
cardinality for all $x\in\Z^3$ with $|x|_\infty\leq N$.
We claim that 
\begin{align} \label{eq:claim}
\begin{split} 
\text{\centerline{ in this case there is a $y\in \Z^3$ with $|y|_{\infty}\geq N$,
such that $y$ is} } \\ 
\text{ \centerline{ connected to $B_{\infty}^{3}(y,|y|_{\infty})^{c}$
by a $*$-path of $u$-bad vertices.}}
\end{split}\end{align}
Let us for a moment assume that the claim is correct. Then by 
Proposition \ref{prop:badComponentsSmall} and using
a union bound in combination with the fact that the $ \vert B_\infty^3 (y,k) \vert \le 6(2k+1)^2$ elements, one has
\begin{equation}
\label{BNnotlinked}
\begin{aligned}
\P\Big[\mathcal{G}_x \mbox{ is finite for all $|x|_\infty\leq N$}\Big] 
\leq C_2'\sum_{k=N}^{\infty} k^2e^{-k^{C_3}},
\end{aligned}
\end{equation}
which is smaller than one if $N$ is large enough.
Consequently there is, with positive $\P$-probability, 
an infinite connected component in $\mathcal{G}^{u}$.
Since the existence of an infinite connected component in $\mathcal{G}^{u}$ is an event that is
invariant under shifts in
$\Z^3$, and since $\P$ is ergodic with respect to these shifts (see \cite[Theorem 2.1]{SZ10}),
we obtain that, $\P$-a.s. there is an infinite connected component
 in $\mathcal G^u$.\\
We now prove \eqref{eq:claim}. 
If $\partial_{\mathrm{ext}}\mathcal G_x = \emptyset$ for all $x \in B_\infty^3(0,N),$ then due
to the finiteness assumption on the $\mathcal G_x$
we get  $\mathcal G_x = \emptyset$ for all such  $x,$ and hence all such $x$ are $u$-bad,
which would yield the claim.
Therefore, assume otherwise, and let $y'\in
\partial_{\mathrm{ext}}\mathcal G_x,$ with $|x|_{\infty}\leq N,$
be such that it has
maximal first coordinate among all such vertices $y'$ fulfilling $y_2', y_3' \in [-N,N]$
(where $y_i'$, $i\in\{1,2,3\}$, denotes the $i$-th
coordinate of $y'$).
Fix $x$ such that 
$y' \in \partial_{\mathrm{ext}}\mathcal G_{x}$ and denote it by $x^{(0)}.$
We aim to find a $u$-bad $z\in\Z^3$, such that $|z-y'|_{\infty}> |y'|_{\infty}$
and such that there is a $*$-path of $u$-bad vertices which connects $y'$ to $z$.
For this purpose, we distinguish two cases:

{\em (i)}$\;$ If $|y'|_{\infty}\leq N$, then observe that as a consequence of the definition of  $y',$
all vertices in $\partial_{\mathrm{int}}B_{\infty}^{3}(0,N) \cap (\{N\} \times \Z^2)$ are $u$-bad.
Hence, one can immediately choose $y,z \in \partial_{\mathrm{int}}B_{\infty}^{3}(0,N) \cap (\{N\} \times \Z^2)$
fulfilling the required properties.

{\em (ii)}$\;$ Assume now that $|y'|_{\infty}>N.$
Then we have $|y'|_{\infty}= y'_1>N,$
and we set $y^{(0)} := y := y'$ and $x^{(0)} := x.$ 
Define a nearest neighbor path via
$\Phi(n)= (y^{(0)}_1-n,y^{(0)}_2,y^{(0)}_3)$ for $n \ge 0.$ 
In addition
let $z^{(0)}\in\partial \mathcal{G}_{x^{(0)}}$ be 
such that there is no vertex in ${\rm range}(\Phi) \cap \partial \mathcal G_{x^{(0)}}$ 
which has a smaller first coordinate than $z^{(0)},$ 
and define $m_0$ via $\Phi(m_0) := z^{(0)}.$
In particular, by definition we have $z^{(0)} \in\partial_{\mathrm{ext}}\mathcal{G}_{x^{(0)}}.$
In addition, let
$$
n_0=\max\{n\ge m_0 \, :\, \Phi(m) \mbox{ is $u$-bad for all }m_0 \le m \le n \} \wedge 2y_1^{(0)},
$$
and set $y^{(1)} = \Phi(n_0).$
By Tim\'{a}r \cite[Lemma 2]{T13}, the set $\partial_{\mathrm{ext}} \mathcal{G}_{x^{(0)}}$ 
is $*$-connected.
Now if  $y^{(1)}_1 < 0$, then this $*$-connectivity of $\partial_{\mathrm{ext}}\mathcal{G}_{x^{(0)}}$
is enough to deduce the claim.
In fact, in this case we may connect $y^{(0)}$ to $y^{(1)}$ via a
$*$-path of $u$-bad vertices of length more than $\vert y^{(0)} \vert,$ 
by first connecting $y^{(0)}$ to $z^{(0)}$ via a $*$-path contained in $\partial_{\mathrm{ext}} \mathcal{G}_{x^{(0)}}$ and 
by then connecting  $z^{(0)}$ to $y^{(1)}$ along $\Phi$; this would finish the proof.
If, on the other hand, $y_1^{(1)}\ge 0$, then observe that $y^{(1)}\in\partial_{\mathrm{ext}}\mathcal{G}_{\Phi(n_0+1)}$
(to see this,
use that $y^{(1)} \in \partial \mathcal{G}_{\Phi(n_0+1)}$
and that it is connected to $y^{(0)}$ along a $*$-path of $u$-bad vertices, and that $y^{(0)}$ has maximal first coordinate
among all elements 
$z \in \partial_{\mathrm{ext}} \mathcal{G}_{x},$ for some $|x|_{\infty}\leq N,$ 
and such that $z_2,z_3 \in [-N,N]$).
We can now repeat the procedure started in $(ii)$ with $y^{(1)}$ taking the role of $y^{(0)}$ 
in order to obtain 
a $y^{(2)}$ taking the role of the previous $y^{(1)},$ and so on.
I.e., we construct a sequence $y^{(0)},y^{(1)},\ldots, y^{(n)}$ (up to the smallest $n \in \N$ 
such that $y^{(n)}_1 < 0$)
such that $y^{(k)}$
may be connected to $y^{(k+1)}$ by a $*$-path of $u$-bad vertices
for all $k \in \{0, 1, \ldots, n-1\}.$ In particular, since $y^{(k)}_1\leq y^{(k-1)}_1-2$ for all $k\leq n,$
after at most $|y'|_{\infty}/2+1$ iterations (and taking the loop-erasure of the path connecting
$y^{(0)}$ to $z:=y^{(n)}$) we will have found the desired $z\in\Z^3$,
which finally yields the claim.

\end{proof}

\subsection{Proof of Proposition \ref{prop:badComponentsSmall}}
\label{S4.2}
The proof will be divided into several lemmas. For this purpose
 fix $\varepsilon\in(0,1)$
and define 
\begin{equation} \label{eq:uTildeDef}
\widetilde{u}_0=(1-\varepsilon)u_*(d).
\end{equation}
The following estimate will serve as a seed estimate for the decoupling inquality
of Proposition \ref{prop:decoupl} and as such be employed
in Lemma \ref{lem:decoupl}.
\begin{lemma}
\label{lem:G}
There is $d_0\in\N$ such that 
for all $y\in\Z^3$,
\begin{equation*}
\P\Big[\mathcal{G}_{y,\widetilde{u}_0}^{c}\Big]\leq d^{-7}/5 \qquad \mbox{for all } d\geq d_0,\, d\in\N,
\end{equation*}
where $\mathcal{G}_{y,\widetilde{u}_0}$ was defined in \eqref{eq:Gyu}.
\end{lemma} 
\begin{proof}
We will derive the result using Theorem 4.2 of \cite{SZ11}.
For this purpose identify $\Z^2$ with $\Z^2\times\{0\}^{d-2}$ and set
\begin{equation*}
\widetilde{\Z}^2 = \big\{x\in\Z^d\, \colon\, 
x_i=0 \mbox{ for all }i\notin\{2,3\}\big\}.
\end{equation*}
Furthermore, define $\mathcal{G}_{y,{u}}^{2}$ and 
$\widetilde{\mathcal{G}}_{y,u}^{2}$, respectively,
by $\mathcal{G}_{y,u}$ as in \eqref{eq:Gyu},
but with $\Z^3$ replaced by $\Z^2$ and  $\widetilde{\Z}^2$, respectively.
By Remark \ref{rm:unqiue}, there is $d_0\in\N$ such that for $d\geq d_0$,
the hypercube $\mathcal C$ contains at most one connected component $\mathfrak{C}$
with $|\overline{\mathfrak{C}}\cap \mathcal C|\geq (1-d^{-2})|\mathcal C|$.
As a consequence we deduce
\begin{equation}
\label{intersectG}
\mathcal{G}_{0,\widetilde{u}_0}^{2}\cap\widetilde{\mathcal{G}}_{0,\widetilde{u}_0}^{2}
\subseteq \mathcal{G}_{0,\widetilde{u}_0}.
\end{equation}
Finally, it remains to apply Theorem 4.2 in \cite{SZ11}. Note that the intensity parameter in that result
equals $(1-\varepsilon)g(0) \log d$, where $g(0)$ was the Green function at the origin; however
since the main result of \cite{SZ11B} supplies us with $u_*(d) \le (1+\varepsilon)\log d$
for $\varepsilon > 0$ arbitrary $d$ large enough,
and since $g(0) \to 1$ as $d \to \infty$ (see e.g. Lemma 1.2 in \cite{SZ11}),
we can apply it with intensity $\widetilde u_0$ also, if $d$ large enough. Hence, we infer that for $d\geq d_0$
\begin{equation*}
\P\Big[\mathcal{G}_{0,\widetilde{u}_0}^{2}\Big] \geq 1-d^{-7}/10.
\end{equation*}
As the same is true for $\widetilde{\mathcal{G}}_{0,\widetilde{u}_0}^{2}$, in combination with (\ref{intersectG})  we obtain the claim for $y=0$.
Since $\P$ is invariant under shifts in space, we obtain the result for every $y\in\Z^3$.
\end{proof}

\medskip\noindent
We define for $x\in\Z^3$ and $L\geq 1$, $L$ integer,
\begin{equation*}
\begin{aligned}
 A_{x,L}=
\Big\{&\Psi\in\{0,1\}^{\Z^d}\colon\, B_{\infty}^3(x,L) \mbox{ is connected to } 
\partial_{\mathrm{int}} B_{\infty}^3(x,2L)\\
&\mbox{ by a $*$-path on $\Z^3$ along which $\Psi$ equals one}\Big\}.
\end{aligned}
\end{equation*}

If $x\notin \Z^3$, then $A_{x,L}=\emptyset$.
We will denote ``bad'' crossing events by
\begin{equation*}
\begin{aligned}
B^u_{x,L}&=
\Big\{\omega \in \Omega\colon\, \one_{\mathcal{B}^{u}(\omega)}\in A_{x,L} \Big\}\\
&=\Big\{\omega \in \Omega\colon\, B_{\infty}^3(x,L) \mbox{ is connected to 
$B_{\infty}^3(x,2L)$ by a $*$-path on $\Z^3$ of $u$-bad vertices}\Big\},
\end{aligned}
\end{equation*}
where we recall that $\mathcal B^u$ had been defined in 
\eqref{eq:BG}. Also, recall Definition \ref{def:cascade} of 
cascading events.
\begin{lemma}
\label{lem:cascade}
$\mathcal{A} = (A_{x,L})_{x\in\Z^d, L\geq 1 \mbox{ integer}}$ is a family of increasing events 
which cascades with complexity at most $3$.
Moreover $C_1=C_1(\mathcal{A},3)$ as introduced in Definition \ref{def:cascade}
does not depend on $d$.
\end{lemma}
\begin{proof} 
The proof is similar to the proof of 
(3.10) in \cite{SZ12}, except that one additionally has to make use of the fact
that $|\cdot|_2\leq \sqrt{d}|\cdot|_\infty$. We omit the details.\\
\end{proof}
The family of events $(B^u_{x,L})_{x\in\Z^d, L\geq 1 \mbox{ integer}}$ is shift invariant in the following sense: 
Let 
\begin{equation*}
\omega =  \sum_{i\geq 0} \delta_{(w_i^*,u_i)} \in \Omega,
\end{equation*}
and define
\begin{equation*}
\tau_x:\Omega\mapsto \Omega, \quad
\omega \mapsto \sum_{i\geq 0} \delta_{(w_i^*+x,u_i)},
\end{equation*}
where $w^* + x = \pi(w(\cdot)+x)$, any $w\in \pi^{-1}(w^*)$.
Then for all $x,y\in\Z^3$ one has
\begin{equation}
\label{eq:Ashiftinv}
\omega \in B_{x,L} \quad \mbox{if and only if}\quad \tau_y(\omega) \in B_{x+y,L}.
\end{equation}
We are now in the position to apply the decoupling inequality (\ref{eq:decoupl}).
For this purpose $l_0$ and $L_0$ are such that they satisfy the relations
\begin{equation}
\label{eq:fixl0L0}
l_0\geq 10^6\sqrt{d}c_0 \qquad \mbox{and} \qquad L_0=\lceil \sqrt{d} \rceil.
\end{equation}
We further recall the definition of $u_{\infty}^{-}$,
see the lines following (\ref{eq:fdecoupl}), as well as the definition of $\widetilde u_0,$ in \eqref{eq:uTildeDef}.
\begin{lemma}
\label{lem:decoupl}
 There is $d_0\in\N$ such that for all $d\geq d_0$, $d\in\N$, 
there is $l_0$ satisfying \eqref{eq:fixl0L0} such that for all $u\leq u_{\infty}^{-}=u_{\infty}^{-}(\widetilde{u}_0),$
one has
\begin{equation}
\label{decouplA}
\P\big[B^u_{0,L_n}\big]\leq e^{-2^n}.
\end{equation}
\end{lemma} 
\begin{proof}
By Proposition \ref{prop:decoupl}, Lemmas \ref{lem:G}--\ref{lem:cascade},
(\ref{eq:Ashiftinv}) 
and the fact that $\P$ 
is invariant under shifts in $\Z^3$, we get
\begin{equation}
\label{eq:Bestimate}
\begin{aligned}
\P\Big[B_{0,L_n}^{u_{\infty}^{-}}\Big] 
&\leq \Big(C_1 l_0^6\Big)^{2^n}\Big(\P\big[B_{0,L_0}^{\widetilde{u}_0}\big] 
+ \varepsilon(u_{\infty}^{-})\Big)^{2^n}.
\end{aligned}
\end{equation}
To estimate the probability on the right-hand side of \eqref{eq:Bestimate}, note that
\begin{equation*}
\begin{aligned}
\P\big[B_{0,L_0}^{\widetilde{u}_0}\big]
\leq \P\big[\text{there is $x\in\partial_{\mathrm{int}}^{*}B_{\infty}^{3}(0,L_0)$
which is $\widetilde{u}_0$-bad}\big]
\leq cd\P\Big[\mathcal{G}_{y,\widetilde{u}_0}^{c}\Big]
\leq cd^{-6},
\end{aligned}
\end{equation*}
where we used a union bound in combination with the fact that there is a constant
$c>0$ such that the cardinality of $\partial_{\mathrm{int}}^{*}B_{\infty}^{3}(0,L_0)$ is bounded by $cd$ to get the second inequality. The last inequality is a consequence of Lemma \ref{lem:G}.

Hence, in order to prove the desired decay of the right-hand side, it is enough to
determine $l_0$ such that
\begin{equation}
\label{L0l0}
\begin{aligned}
 cC_1l_{0}^6d^{-6}\leq \frac{1}{2e}\qquad \mbox{ and}
\qquad C_1l_0^6\varepsilon(u_{\infty}^{-})\leq \frac{1}{2e}.
\end{aligned}
\end{equation}
The first inequality in (\ref{L0l0}) is indeed satisfied for all $d$ large enough, subject to the choice of $l_0$ in \eqref{eq:fixl0L0}.
To show the second inequality in (\ref{L0l0}), observe that
\begin{equation*}
\lim_{d\to\infty}\frac{c_1^d}{l_0^{(d-3)/2}} = 0,
\qquad \mbox{for }  l_0\geq 10^6\sqrt{d}c_0.
\end{equation*}
Employing this equality in the definition of $u_{\infty}^-$ in
(\ref{eq:fdecoupl}), we obtain that  $u_{\infty}^{-}\geq (1-2\varepsilon)u_*(d)$, if $d$ large enough.
Using this inequality and the fact that by the main result of \cite{SZ11} one has that for $d$ large enough $u_*(d)\geq (1-\varepsilon)\log d$, the definition of $\varepsilon(u_{\infty}^{-})$ leads to the desired estimate.
This shows that (\ref{decouplA}) is true for $u=u_{\infty}^{-}$. 
The claim for every other $u\leq u_{\infty}^{-}$ follows by the fact that $B_{x,L}^u$
is increasing
in $u.$
\end{proof}
As a direct consequence of this result we obtain the following corollary.
\begin{corollary}
\label{cor:bad}
If \eqref{decouplA} holds true, then for some $C=C(d)< \infty$, 
all $u\leq u_{\infty}^{-}$ and $N\geq 1,$
\begin{equation*}
\begin{aligned}
&\P\Big[\mbox{There is a $*$-path of $u$-bad vertices
from the origin to $\partial_{\mathrm{int}} B_{\infty}^3(0,N)$}\Big]
 \leq Ce^{-N^{\frac{1}{C}}}.
\end{aligned}
\end{equation*}
\end{corollary}
Using this corollary, we can now prove Proposition \ref{prop:badComponentsSmall}.
\begin{proof}[Proof of Proposition \ref{prop:badComponentsSmall}]
For all $l_0$ subject to \eqref{eq:fixl0L0} using similar arguments as in the proof of Lemma \ref{lem:cascade} we see that there is $d_0$
such that for all $d\geq d_0$ one has $u_{\infty}^{-} \geq (1-2\varepsilon)u_{*}(d)$.
Fix $u\leq u_{\infty}^{-}$.
Due to the shift-invariance of $\P$ it is enough to prove the result for $x=0$.
Assume that $0$ is in a $*$-connected component of $u$-bad vertices of length
at least $N$. Consequently,  there is a $*$-path of $u$-bad vertices
from $0$ to $\partial_{\mathrm{int}} B_{\infty}^3(0,(N/3)^{1/d})$ in $\Z^3$. 
Thus, by Corollary \ref{cor:bad},
\begin{equation*}
\begin{aligned}
&\P \big[ 0 \text{ is contained in a $*$-path 
of $u$-bad vertices of length at least } N \big]\\
&\le \P\Big[\mbox{There is a $*$-path of $u$-bad vertices
from the origin to $\partial_{\mathrm{int}} B_{\infty}^3(0,cN^\frac1d)$}\Big]\\
& \leq C\sum_{k=N/3}^{\infty}e^{-k^{1/(Cd)}}
 \leq C_2e^{-N^{C_3}}, \ C_2,C_3 > 0,
\end{aligned}
\end{equation*}
which proves the claim.
\end{proof}


\section{Proof of Proposition \ref{prop:transofG} (transience of $\mathcal G^u_\infty$)}
\label{S5}

In this section we take advantage of the relations between simple random walk and
electrical network theory in order to deduce 
that $\mathcal{G}_{\infty}^{u}$ is transient for $u$ as in 
\eqref{prop:badComponentsSmall} and $d$ large enough (see Proposition \ref{prop:transofG}).
\subsection{Rerouting paths around bad vertices} \label{sec:rerouting}
\label{S5.1}
The following is inspired by methods of \cite{AnBeBePe-06}.
Assume that the almost sure event of Proposition \ref{prop:infiniteGoodness}
occurs. Since $\Z^3$ is transient, Lemma \ref{lem:transienceCharact} supplies us with the existence of a probability
measure $\mu$ on infinite simple nearest neighbor paths in $\Z^3$ starting in some $y \in \Z^3,$ and
fulfilling 
\eqref{eq:finiteExpVerticesIII}. 
The idea now is to map infinite simple nearest neighbor paths $\pi$ on $\Z^3$
via a function $\widehat \varphi$
to infinite simple nearest neighbor paths $\widehat{\varphi}(\pi)$ on $\mathcal{G}_{\infty}^{u} \subset \Z^3$
in such a way that $\mu \circ \widehat \varphi^{-1}$ still satisfies condition
\eqref{eq:finiteExpVerticesIII} and hence, again by Lemma \ref{lem:transienceCharact},
this supplies us with the transience of $\mathcal G_\infty^u.$
This mapping will be constructed by cutting out pieces of a path $\pi$ on $\Z^3$ which are not in 
$\mathcal{G}_{\infty}^{u}$ and afterwards replacing them by finite simple nearest neighbor paths
of vertices on $\partial_{{\rm int}}\mathcal{G}_{\infty}^{u}$.
These sequences are chosen in such a way that they connect all parts 
of the path which are inside $\mathcal{G}_{\infty}^{u}$.
In order to ensure that $\P$-a.s., the measure $\mu \circ \widehat \varphi^{-1}$ still satisfies condition
\eqref{eq:finiteExpVerticesIII}, we will have to ensure that 
$*$-connected components of $u$-bad vertices are not too large. 
This is the content of the following lemma.

\begin{lemma}
\label{lem:finitecomponent}
Let $u$ and $d$ 
be as in Proposition \eqref{prop:badComponentsSmall}. Then there is $C_4 >0$ such that $\P$-a.s. one finds $N_0\in\N$ such that
for all $N\geq N_0$ the event
\begin{equation*}
\begin{aligned}
\Big\{\mbox{there is } x\in B_{\infty}^{3}(0,N) &\text{ such that } 
 x \mbox{ is contained in a simple $*$-path}\\
& \mbox{of $u$-bad vertices of length at least $(\log N)^{C_4}$}\Big\}
\end{aligned}
\end{equation*}
does not occur.
\end{lemma}
\begin{proof}
This follows from Proposition \ref{prop:badComponentsSmall} and
an application of the Borel-Cantelli Lemma.
\end{proof}
In the rest of this section we describe the mapping $\widehat \varphi$ that will send
infinite simple nearest neighbor paths on $\Z^3$ to
infinite simple nearest neighbor paths $\widehat{\varphi}(\pi)$ on $\mathcal G_\infty^u$
as alluded to above. 
Let $\pi$ be an infinite simple nearest neighbor path on $\Z^3$.
We use the following
notation for the sequence of successive returns to and departures from $\mathcal G_{\infty}^{u}$:
\begin{equation*}
\begin{aligned}
D_0&= \min \big\{ k \geq 0\colon\,  \pi(k) \in \Z^3\setminus\mathcal G_\infty^u\big\},
\ R_0= \min \big\{k>D_0\colon\, \pi(k) \in \mathcal G_\infty^u\big\},\\
D_n&= \min \big\{k>R_{n-1}\colon\, \pi(k) \in \Z^3\setminus\mathcal G_\infty^u\big\}, 
\ R_n= \min \big\{k>D_n\colon\, \pi(k) \in \mathcal G_\infty^u\big\},
 \mbox{ for }n\in\N.
\end{aligned}
\end{equation*}
We modify the path $\pi$  on $\Z^3$ in the following way:
\begin{enumerate}

\item if $D_0 = 0,$ we erase the segment $(\pi(0), \ldots, \pi(R_0-1));$

\item
for each $n$ with  $0 < D_n < \infty$ we replace the segment $(\pi(D_{n}), \dots, \pi(R_{n}-1))$
by a finite shortest simple nearest neighbor path on $\mathcal{G}_{\infty}^{u}$ which 
connects $\pi(D_n-1)$ to $\pi(R_{n})$.

\end{enumerate}

Finally, let $\widehat{\varphi}(\pi)$ be the loop-erasure of the path obtained this way,
which is an infinite simple nearest neighbor path on $\mathcal{G}_{\infty}^{u}$.
Below we will use the notation
\begin{align*}
\mathcal B_{x,u}= \left\{
 \begin{array}{ll}
\text{the $*$-connected  component of $x \in \Z^3 \backslash \mathcal G^u_\infty$ of $u$-bad vertices,} &
\text{ if $x$ is $u$-bad}, \\
\emptyset, & \text{ if $x$ is $u$-good.}
\end{array}
\right.
\end{align*}

\begin{remark}
\label{rem:construction}
Step $(b)$ in the above construction is $\P$-a.s. well-defined. In fact, if $D_n < \infty,$ then
by Lemma \ref{lem:finitecomponent},
 $\mathcal{B}_{\pi(D_n),u}$ 
is of finite cardinality, and $\pi$ has to hit $\partial_{\mathrm{ext}}^{*} \mathcal{B}_{\pi(D_n),u}$ in finite time.
 By definition, $\partial_{\mathrm{ext}}^{*} \mathcal{B}_{\pi(D_n),u}$
 consists of $u$-good vertices only; in addition,
 due to \cite[Theorem 4]{T13}, it
is connected, and since it contains $\pi(D_n-1) \in \mathcal G_\infty^u,$
we get $\partial_{\mathrm{ext}}^{*} \mathcal{B}_{\pi(D_n),u}
\subset \mathcal{G}_{\infty}^{u}$. As a consequence,  $R_n,$ $n\geq 1,$ coincides with the first hitting time of 
$
\partial_{\mathrm{ext}}^{*} \mathcal{B}_{\pi(D_n),u}
$
after time $D_n$ and is finite. If $D_0>0$, then the same arguments show that $R_0$ is finite. To see that this is also true in the case that
$D_0=0$ note that one may connect $\pi(D_0)$ by a finite
nearest neighbor path to $\mathcal{G}_{\infty}^{u}$. This allows to apply the previous arguments to deduce the finiteness of $R_0$ also in this case.
In particular, a finite shortest simple nearest neighbor path as postulated in $(b)$
exists.
\end{remark}


\subsection{Rerouting paths preserves finite energy}
\label{S5.2}
In this section we show that $\widehat{\varphi}(\pi)$ induces a probability measure as in
 condition $(b)$ of Lemma
\ref{lem:transienceCharact}.
In fact, since
$\Z^3$ is transient,
Lemma \ref{lem:transienceCharact} implies that
there is $z\in\Z^3$ and a probability measure $\mu$ on  $\Pi_z(\Z^3)$ which satisfies the finite energy condition 
\eqref{eq:finiteExpVerticesIII}, i.e.,
 we have 
\begin{equation} \label{eq:finiteEnergy}
\sum_{x \in \Z^3} \mu^2[\pi \in \Pi_z(\Z^3) \, : \, x \in \pi] < \infty.
\end{equation}

By Lemma \ref{lem:transienceCharact}, in order to prove that $\mathcal G^u_\infty$ is transient
a.s.,
we only need to show that $\mu \circ \widehat{\varphi}^{-1}$ 
satisfies \eqref{eq:finiteExpVerticesIII}, i.e., 
we have 
\begin{equation} \label{eq:finitenessToShow}
\sum_{x \in \Z^3} \mu^2[x\in \widehat{\varphi}(\pi)] < \infty, \quad \P-\text{a.s.}\footnote{In fact, note that since 
by Lemma \ref{lem:finitecomponent} we have $\vert \mathcal B_{z,u} \vert < \infty$ a.s., there exists $z' \in \mathcal G_\infty^u$
such that $\mu \circ \widehat \varphi^{-1}$ puts positive mass on 
$\Pi_{z'}(\mathcal G_\infty^u).$ Restricting $\mu$ to this latter set and normalizing it puts us into the exact context of
Lemma \ref{lem:transienceCharact}.}
\end{equation}
We set for $x, y \in \Z^3$
\begin{align*}
S(x)= \left\{
 \begin{array}{ll}
\partial_{\mathrm{int}}^* \mathcal \Z^3 \backslash \mathcal B_{x,u}, & 
\text{ if } x \in \Z^3 \backslash \mathcal G^u_\infty, \\
\{x\}, &
 \text{ if } x \in \mathcal G^u_\infty,
\end{array}
\right.
\end{align*}
and
\begin{equation*} 
 T(y)=\{ x\,\colon\,  y \in S(x) \}.
\end{equation*}

Using the definition of $\widehat \varphi,$ we obtain the first inequality in
\begin{equation*}
\mu^2[x \in \widehat{\varphi}(\pi)] \leq \left(\sum_{ y \in T(x) } \mu[y \in \pi] \right)^2
\leq
|T(x)| \sum_{ y \in T(x) } \mu^2[y \in \pi],
\end{equation*}
and the second inequality in this chain is due
to the Cauchy-Schwarz inequality.
Hence, 
\begin{equation*}
 \sum_{x \in \Z^3} \mu^2[x \in \widehat{\varphi}(\pi)] 
\leq \sum_{x \in \Z^3} |T(x)| \sum_{ y \in T(x) } \mu^2[y \in \pi]
= \sum_{y \in \Z^3} \mu^2[y \in \pi]  \sum_{ x \in S(y) }  |T(x)| .
\end{equation*}
Therefore, in order to establish
\eqref{eq:finitenessToShow},  by 
(\ref{eq:finiteEnergy}) it suffices 
to show that 
\begin{equation}\label{expect_S_T_finite}
 \sup_{ x \in \Z^3} \E \left[ \sum_{ y \in S(x) }  |T(y)| \right] <\infty.
\end{equation}

\begin{lemma}
\label{lem:finiteenergy}
The term in (\ref{expect_S_T_finite}) is finite.
\end{lemma}
\begin{proof}
By shift invariance of $\P$ it suffices to prove the claim for $x=0$.
Note that
\begin{equation*}
\begin{aligned}
z\in \bigcup_{y\in S(0)} T(y) 
\Longleftrightarrow S(0)\cap S(z)\neq \emptyset,
\end{aligned}
\end{equation*}
which yields 
\begin{equation}
\label{expectT}
\E\left[\sum_{y\in S(0)} |T(y)|\right] 
= \P[S(0)\neq \emptyset] + \sum_{z\neq 0} \P[S(0)\cap S(z) \neq \emptyset].
\end{equation}
To estimate the second term on the right-hand side of (\ref{expectT})
note that if $S(0)\cap S(z)\neq \emptyset$, then for $y \in S(0)\cap S(z)$, 
\begin{equation}
\label{eq:S0Sz}
\begin{aligned}
&(1)\,  \mbox{ there is } x_0\in\mathcal{B}_{0,u}
\mbox{ such that } \vert y-x_0\vert_{\infty}=1;\\
&(2)\,  \mbox{ and there is } x_1\in\mathcal{B}_{z,u}
\mbox{ such that } \vert y-x_1\vert_{\infty}=1.
\end{aligned}
\end{equation}
Since $\mathcal{B}_{0,u}$ and $\mathcal{B}_{z,u}$ are $*$-connected, there is a $*$-path
of $u$-bad vertices starting in $0$ and ending in $x_0$, and a $*$-path of $u$-bad vertices starting in $x_1$ and ending in $z$.
Since $\vert x_0-x_1 \vert_\infty \le 2,$ we infer that
at least one of these two paths must have length at least
$\floor{\vert z\vert_{\infty}-1}/2$, and hence either $0$ is contained in a $*$-path of $u$-bad vertices
of length at least $\floor{\vert z\vert_{\infty}-1}/2$, or this property holds for $z$.
Proposition \ref{prop:badComponentsSmall} and the shift invariance of $\P$ yield
\begin{equation}
\label{estS}
\begin{aligned}
&\P[S(0)\cap S(z) \neq \emptyset]\\
&\qquad \leq 2
\P\big[0 \mbox{ is contained in a simple $*$-path of $u$-bad vertices of length at least $\floor{\vert z\vert_{\infty}-1}/2$}\big]\\
&\qquad \leq C_5e^{-|z|_{\infty}^{C_6}}, \qquad C_5, C_6>0.
\end{aligned}
\end{equation}
\end{proof}



\appendix
\section{Proof of Proposition \ref{prop:decoupl}}
\label{A}
In this appendix we prove Proposition \ref{prop:decoupl}.
The proof is essentially the same as the proof of Theorems 2.1 and 3.4 in \cite{SZ12}.
While the proof of the latter one goes through in exactly the same way, 
we restrict ourselves to giving the main modifications of the proof of
Theorem 2.1 in \cite{SZ12}.
Note that the setting in \cite{SZ12} differs sligthly from the setting
of the current work. Indeed, in \cite{SZ12} more general graphs are considered
and the norm in \cite{SZ12} is different from the Euclidean norm we are considering here. Nevertheless, as stated in the first paragraph in \cite{SZ12}
up to a change of constants the results of \cite{SZ12}
stay true when working in the setting of this article.

\medskip\noindent
$\bullet$ {\bf Notation in \cite{SZ12}.}
Let $l_0>1$ be a constant to be chosen later on, $L_0\geq 1$, 
and define  the geometric scales $L_n=l_0^nL_0$, $n\in\N_0$ .
For $n\in\N_0$, we denote the dyadic tree of depth $n$ by 
$T_n=\bigcup_{0\leq k \leq n} \{1,2\}^k$ and the set of vertices of the tree
at depth $k$ by $T_{(k)} = \{1,2\}^k$.
Given a mapping $\T\colon T_n\rightarrow \Z^d$, we define
\begin{equation}
\label{tau}
x_{m,\T} = \T(m), 
~\widetilde{C}_{m,\T} = B_2(x_{m,\T}, 10\sqrt{d} L_{n-k}), 
\mbox{ for } m\in T_{(k)},~ 0\leq k \leq n.
\end{equation}
For any $0\leq k < n$, $m\in T_{(k)}$, we say that $m_1, m_2$
are the two descendants of $m$ in $T_{(k+1)}$, if they are obtained
by concatenating $1$ and $2$ to $m$, respectively.
We say that $\T$ is a permitted embedding if for any $0\leq k<n$
and $m\in T_{(k)}$,
\begin{equation}
\label{admiss}
\widetilde{C}_{m_1,\T} \cup \widetilde{C}_{m_2,\T} \subseteq \widetilde{C}_{m,\T},
\qquad |x_{m_1,\T}-x_{m_2,\T}|_2\geq  \frac{\sqrt{d}}{100}L_{n-k}.
\end{equation}
The set of all permitted embeddings is denoted by $\Lambda_n$.
Given $n\in\N_0$ and $\T\in \Lambda_n,$ we say that a family  $A_m$, $m\in T_{(n)},$ of events
 of measurable subsets
of $\{0,1\}^{\Z^d}$
is  $\T$-adapted if
\begin{equation*}
A_m \mbox{ is } \sigma\big(\Psi_x, x\in \widetilde{C}_{m,\T}\big)-
\mbox{measurable for each } m\in T_{(n)}.
\end{equation*}
For $n\in\N_0$ and $\T\in \Lambda_{n+1}$, we denote by $\T_i$, $i\in\{1,2\}$,
the embeddings of $T_n$ such that $\T_i(m) = \T((i,i_1,\ldots,i_k))$,
for $m=(i_1,i_2,\ldots, i_k)$ in $T_{(k)}$.
Given a $\T$-adapted collection $A_m$, $m\in T_{(n+1)}$,
we define $\T_i$-adapted collections, $A_{m,i}$, $i\in\{1,2\}$, via
\begin{equation*}
A_{m,i} = A_{(i,i_1,\ldots,i_n)}, 
\mbox{ for } m=(i_1,i_2,\ldots,i_n)\in T_{(n)} .
\end{equation*}

\medskip\noindent
$\bullet$ {\bf The Proof.}
Recall (\ref{eq:fdecoupl}--\ref{eq:epsilondec}) and the convention we made about constants
in the introduction. We now adapt Theorem 2.1 in \cite{SZ12} to our setting.

\begin{theorem}
\label{thm:indstep}
There are $c_0,c_1 >1$,
such that for  $l_0 \geq 10^6\sqrt{d}c_0$ and $L_0\geq \sqrt{d}$,
for all $n\in\N_0$, $\T\in \Lambda_{n+1}$,
for all $\T$-adapted collections $A_m$, $m\in T_{(n+1)},$ 
of increasing events on $\{0,1\}^{\Z^d}$, and for all
$u > u'> 0$ such that
\begin{equation*}
u\geq \Bigg(1+32e^2c_1^d\frac{1}{(n+1)^{3/2}}l_0^{-(d-3)/2}\Bigg)u',
\end{equation*}
one has
\begin{equation*}
\P\left[\bigcap_{m\in T_{(n+1)}} A_{m}^{u'}\right] 
\leq \P\left[\bigcap_{\overline{m}_1\in T_{(n)}} A_{\overline{m}_1,1}^{u}\right]
 \P\left[\bigcap_{\overline{m}_2\in T_{(n)}} A_{\overline{m}_2,2}^{u}\right]
 +2e^{-2u'\frac{1}{(n+1)^3}L_n^{d-2}l_0}.
\end{equation*}
\end{theorem}
%

\begin{proof}
The proof is analogous to that of Theorem 2.1 in \cite{SZ12}.
Thus, we only point out the modifications which are necessary to 
adapt the proof of \cite{SZ12} to our setting.\\
First replace Lemma 1.2 in \cite{SZ12}, which is used
in equation (\rm{2.31}) in \cite{SZ12}, by Proposition 1.3
in \cite{SZ11B}, which reads as follows.

\begin{proposition}
\label{prop:harnack}
There exist $c_0, c_4>1$, 
such that if $L\geq d$  and if $h$ is a non-negative function defined on 
$\overline{B_{2}(0,c_0L)}$ and harmonic in $B_{2}(0,c_0L)$, one has
\begin{equation*}
\max_{x\in B_2(0,L)}h(x) \leq c_4^d \min_{x\in B_2(0,L)} h(x).
\end{equation*}
\end{proposition}
 Second, define similarly as in \cite{SZ12}, {\rm (2.13)--(2.14)}, for $i\in\{1,2\}$
and $\T\in \Lambda_{n+1}$

\begin{equation*}
\qquad \ U= U_1\cup U_2 \qquad \mbox{ with } \qquad
 U_i= B_2\bigg(\T(i), \frac{\sqrt{d}L_{n+1}}{1000}\bigg).
\end{equation*}
as well as
\begin{equation}
\label{eq:M}
\qquad \widetilde{B}_i= B_2\bigg(\T(i), \frac{\sqrt{d}L_{n+1}}{2000M}\bigg)
\end{equation}
for a constant $1\leq M\leq l_0/(2\cdot 10^4)$ to be determined.
Note  in particular that, by (\ref{admiss}), one has $U_1\cap U_2 = \emptyset$.
Moreover, from the definition of the scales $L_n$ we infer $\widetilde C_{i,\T} \in U_i,$ $i \in \{1,2\}$.

The forthcoming lemma replaces Lemma 2.3 in \cite{SZ12} 
and provides bounds on the probability
that a random walk starting in $\partial U \cup \partial_{\mathrm{int}} U$
enters a strict subset $\widetilde{W}$ of $U$ in finite time.
It is applied in equations (\rm{2.33}) and (\rm{2.36}) in \cite{SZ12}.
Before stating the lemma we recall the definition of the entrance time $H_K$
in (\ref{eq:hittime}) and we  moreover define
\begin{equation*}
P_{e_U} = \sum_{x\in U} e_{U}(x)P_x.
\end{equation*}

\begin{lemma}
\label{lem:equilbrest}
Let  $l_0\geq 10^6\sqrt{d}c_0$ and $L_0\geq \sqrt{d}$.
For any $\widetilde{W} \subseteq B_2 \big(\T(1), \sqrt{d}L_{n+1}/2000\big)\cup B_2 \big(\T(2), 
\sqrt{d}L_{n+1}/2000\big)$, $x\in \partial U \cup \partial_{\mathrm{int}} U$,
$x'\in \widetilde{W}$,
one has for some constants  $c_5,c_6>0$,
\begin{equation*}
c_5^{d}L_{n+1}^{-(d-2)} e_{\widetilde{W}}(x')
\leq P_x\big[H_{\widetilde{W}} <\infty, X_{H_{\widetilde{W}}} =x'\big]
\leq c_6^dL_{n+1}^{-(d-2)} e_{\widetilde{W}}(x').
\end{equation*}
\end{lemma}
\begin{proof}
The proof follows the lines of the proof of Lemma 2.3 in \cite{SZ12} with a special attention to the dependence of constants
on the dimension.
First, since $\widetilde{W} \subseteq U$, one has the sweeping 
identity
\begin{equation*}
e_{\widetilde{W}}(x') = P_{e_U}\big[H_{\widetilde{W}} <\infty,
X_{H_{\widetilde{W}}} =x'\big],
\end{equation*}
from which one infers that
\begin{equation}
\label{eq:appsweeping}
\begin{aligned}
& \C(U)\inf_{x\in \partial_{\mathrm{int}} U} 
P_x\big[H_{\widetilde{W}} <\infty, X_{H_{\widetilde{W}}} =x'\big]
\leq e_{\widetilde{W}}(x')\\
&\qquad\leq \C(U)\sup_{x\in \partial_{\mathrm{int}} U} 
P_x\big[H_{\widetilde{W}} <\infty, X_{H_{\widetilde{W}}} =x'\big].
\end{aligned}
\end{equation}
Next, we claim that using (\ref{tau}--\ref{admiss}) 
one can find  $c_7>0$ such that
any two points $x_1,x_2 \in \partial_{\mathrm{int}} U$
may be connected by not  more than $c_7$
overlapping balls  $B_2\big(x',\sqrt{d}L_{n+1}/4000c_0\big)$, $x' \in \partial_{\mathrm{int}}U\cup U^c.$ In fact, 
along the lines of Lemma 2.2 of \cite{SZ11B}, any two points on $\partial _{\mathrm{int}} U_i$ can be connected ``along'' the great circle centered in $\T(i)$ with radius 
$\frac{\sqrt{d}L_{n+1}}{1000},$ by $c_7/3$ such overlapping balls; on the other hand, from \eqref{tau} one can deduce that the
same is true for two points $y_1,y_2$ such that $y_i \in \partial_{{\rm int}} U_i,$ and such that they have minimal
distance among any such pair of points, whence the claim follows.
Since the function $h(x)= P_x\big[H_{\widetilde{W}} <\infty, X_{H_{\widetilde{W}}} =x'\big]$
is non-negative and harmonic  on $B_2\big(x',\sqrt{d}L_{n+1}/4000\big) \subseteq \widetilde{W}^c,$ for all $x'\in \partial_{\mathrm{int}}U\cup U^c,$
and since
$\sqrt{d}L_{n+1}/4000c_0 \ge d$,
we obtain by Proposition \ref{prop:harnack} that
\begin{equation}
\label{applharnack}
\begin{aligned}
\sup_{x\in \partial_{\mathrm{int}}U}  
P_x\big[H_{\widetilde{W}} <\infty, X_{H_{\widetilde{W}}} =x'\big]
&\leq   c_4^{dc_7}
\inf_{x\partial_{\mathrm{int}}U}  
P_x\big[H_{\widetilde{W}} <\infty, X_{H_{\widetilde{W}}} =x'\big]\\
&= c^d \inf_{x\in \partial_{\mathrm{int}}U}  
P_x\big[H_{\widetilde{W}} <\infty, X_{H_{\widetilde{W}}} =x'\big].
\end{aligned}
\end{equation}
Finally, note that by \eqref{eq:capBd} and the subadditivity of capacity (see \eqref{eq:subadd})
we have
\begin{equation}
\label{capU}
\Bigg(\frac{c_2L_{n+1}}{1000}\Bigg)^{d-2}
\leq \C(U)
\leq 2(c_3 L_{n+1})^{d-2}.
\end{equation}
Inserting (\ref{applharnack}) and (\ref{capU}) into (\ref{eq:appsweeping}),
yields the claim for $x\in \partial_{\mathrm{int}}U$.
The extension to $x\in\partial U$ follows from the fact that
$P_x[X_1=y]=1/(2d)$ for all $x,y\in\Z^d$ with $|x-y|_1=1$.
\end{proof}
Since for all $n\in\N_0$ the inequality $\sqrt{d}L_n\geq d$ holds one can apply
\eqref{eq:capBd} to all balls in the Euclidean norm whose radius is larger than
$\sqrt{d}L_n$. 
Using this fact repeatedly, from that moment on, the proof works similarly as the proof of \cite[Theorem 2.1]{SZ12}.
In particular $M$ as introduced in \eqref{eq:M}, which is determined in equation (2.36) in \cite{SZ12}, satisfies
\begin{equation*}
c_6^dL_{n+1}^{-(d-2)} \C(\widetilde{B}_1\cup \widetilde{B}_2)
\leq 2c_6^dL_{n+1}^{-(d-2)}\Big(\frac{c_3L_{n+1}}{2000M}\Big)^{d-2}
\leq (2e)^{-1}.
\end{equation*}
Thus, $M$ does not depend on $d$.
To conclude Proposition \ref{prop:decoupl} from Theorem \ref{thm:indstep}
one proceeds as in the proof of \cite[Theorem 3.4]{SZ12}.
\end{proof}

{\bf Acknowledgment:} We are indebted to B. R\'ath and A. Sapozhnikov for inspiring discussions as well as
for useful comments
on a first draft of this paper.

\bibliographystyle{alpha}

\end{document}